 \documentclass[11pt, a4paper, reqno]{amsart}

\usepackage{amsmath,amssymb,amsthm, amsfonts, epsfig}
\usepackage{hyperref}
\usepackage{color}
\usepackage{ulem}
\usepackage{epsfig}

\usepackage{mathrsfs}

\def \R {\mathbb{R}}
\def \supp {\mathrm{supp } }

\def \dist {\mathrm{dist}}

\def \suchthat {\ \big | \ }

\def \Leb {\mathscr{L}^n}

\newtheorem{theorem}{Theorem}[section]
\newtheorem{lemma}[theorem]{Lemma}
\newtheorem{proposition}[theorem]{Proposition}
\newtheorem{corollary}[theorem]{Corollary}
\newtheorem{remark}[theorem]{Remark}

\numberwithin{equation}{section}

\title[The infinity obstacle problem]{Optimal regularity at the free boundary\\ for the infinity obstacle problem}
\author[Rossi, Teixeira and Urbano]{J.D. Rossi, E.V. Teixeira and J.M. Urbano}

\begin{document}

\begin{abstract}
This paper deals with the obstacle problem for the infinity Laplacian. The main results are a characterization of the solution through comparison with cones that lie above the obstacle and the sharp $C^{1,\frac{1}{3}}$--regularity of the solutions at the free boundary.
\end{abstract}

\date{\today}

\keywords{Obstacle problem; infinity Laplacian; free boundary; optimal regularity}

\subjclass[2010]{35B65, 35R35, 35J70}

\maketitle

\section{Introduction}

The regularity of infinity harmonic functions is an outstanding issue in the theory of nonlinear partial differential equations. The belief that viscosity solutions of $\Delta_\infty u =0$ are of class $C^{1,\frac{1}{3}}$ has hitherto remained unproven despite some recent exciting developments. The flatland example of Aronsson
$$u(x,y) = |x|^{\frac{4}{3}} - |y|^{\frac{4}{3}}$$
sets the framework to what can be expected: the first derivatives of $u$ are H\"older continuous with exponent $1/3$, whereas its second derivatives do not exist on the lines $x = 0$ and $y = 0$. The sharpest results to date are due to Evans and Savin, who prove in \cite{Evans_Savin} that infinity harmonic functions in the plane are of class $C^{1,\alpha}$, building upon Savin's breakthrough in \cite{Savin} (the optimal $\alpha$ remains unknown even in 2-D), and to Evans and Smart, who recently obtained in \cite{ES} the everywhere differentiability, irrespective of the dimension.

This paper addresses the obstacle problem for the infinity Laplacian (see \cite{Ju, LM2}) and its most striking results concern the behaviour at the free boundary. We prove, for the zero obstacle problem, that the solution leaves the ground as a $C^{1,\frac{1}{3}}$--function and that this regularity is optimal. The sharp estimates we derive are yet another conspicuous hint towards the optimal regularity for infinity harmonic functions.

As in a number of subfields in the analysis of PDEs, the study of obstacle problems evolved through two parallel paths, namely the variational and the non-variational theories. The former takes into account energy considerations and is driven by elliptic operators in divergence form, while the latter deals with operators in non-divergence form.
In turn, the infinity Laplacian
$$\Delta_\infty u = \sum_{i,j=1}^{d} \frac{\partial u}{\partial x_i} \frac{\partial u}{\partial x_j} \frac{\partial^2 u}{\partial x_i \partial x_j}$$
enjoys a sort of duality character. On the one hand,  it is a genuine degenerate elliptic operator in non-divergence form but, on the other hand,  $\infty$-harmonic functions can be obtained as limits of $p$-harmonic functions, which are solutions to a divergence form equation.  This operator is connected with the optimal Lipschitz extension problem \cite{J}, random tug-of-war games \cite{BEJ, PSSW}, mass transportation problems \cite{GAMPR} and several other applications \cite{cms,MRU}.

\medskip

The variational approach to the obstacle problem for elliptic operators has been extensively studied. The classical setting amounts at minimizing the energy
$$E(u) = \int_\Omega |D u|^2$$
among the functions that coincide with a given function $F$ at the boundary of $\Omega \subset \R^d$ and remain above a prescribed obstacle $\Psi$. Such a problem is motivated by the description of the equilibrium position of a membrane (the graph of the solution) attached at level $F$ along the boundary of $\Omega$ and that is forced to remain above the obstacle in the interior of $\Omega$. The same mathematical framework appears in many other contexts: fluid filtration in porous media, elasto-plasticity, optimal control or financial mathematics, to name just a few.  In the section \ref{S. limit p-obs}, we explore the ``limiting divergence structure" of the infinity Laplacian to introduce the infinity obstacle problem and obtain a solution $u_\infty$, passing to the limit, as $p \to \infty$, in a sequence of solutions $u_p$ to the obstacle problem for the $p$-Laplacian. With the aim of gaining some insight on the problem, a radially symmetric explicit example is studied in an appendix. We then deal with characterizations of the limit. We first show that $u_\infty$ is the smallest infinity superharmonic function in $\Omega$ that is above the obstacle and equals $F$ on the boundary, a result that implies its uniqueness. Then we establish a sort of comparison with cones that lie above the obstacle. This characterization is interesting in its own right but it also implies a regularity result at the free boundary, a warm-up for what will come later. The section closes with the analysis of the behaviour at infinity of the coincidence sets for the $p$-obstacle problem and its relation with the coincidence set of the limiting problem.

\medskip

The heart of the paper is section \ref{batfb}, where the zero-obstacle type problem that views $\Delta_\infty$ as a degenerate elliptic operator in non-divergence form is studied. We establish the optimal asymptotic profile near the free boundary, showing the solution behaves as a $C^{1, \frac{1}{3}}$--function. We use this sharp information to deduce the uniform positive density of the non-coincidence set. In particular, the free boundary does not develop cusps pointing inwards to the coincidence set.

\section{The variational $\infty$--obstacle problem and characterizations of the limit} \label{S. limit p-obs}

Let $\Omega \subset \mathbb{R}^d$ be a bounded smooth domain, $F$ a Lipschitz function on $\partial \Omega$
and $1< p <\infty$. Given an obstacle $\Psi \colon
\overline{\Omega} \to \mathbb{R}$, with
\begin{equation} \label{obst_vs_bdry}
    \sup\limits_{\partial \Omega} \Psi < \inf\limits_{\partial \Omega}  F,
\end{equation}
the $p$-degenerate obstacle problem for $\Psi$ refers to the minimization problem
\begin{equation}\label{p.minimos}
   \text{Min} \left \{ \int_\Omega |D v(x)|^p dx \suchthat
   v \in W^{1,p}_F \text{ and }  v \ge \Psi \right \}.
\end{equation}
Here $W^{1,p}_F$ means the set of functions in $W^{1,p} (\Omega)$
with trace $F$ on $\partial \Omega$.

Simple soft functional analysis arguments assure that
\eqref{p.minimos} has a unique solution $u_p$.  Let $z$ be a
Lipschitz extension of $F$ such that $z\geq \Psi$ (for the proof
of the existence of such $z$ see Proposition \ref{coimbra}). Since
$z$ competes in the minimization problem
\eqref{p.minimos} for every $p$, we have
$$ \left(\int_\Omega |D u_p |^p\right)^{1/p}  \leq L |\Omega|^{1/p},$$
where $L:=\| D z \|_{L^\infty (\Omega)}$.
For a fixed $q$ and $p \geq q$, we can write
$$
    \left(\int_\Omega |D u_p |^q\right)^{1/q} \leq
    \left(\int_\Omega |D u_p |^p\right)^{1/p}
    |\Omega|^{\frac{p-q}{pq}} \leq L |\Omega|^{1/p} |\Omega|^{\frac{p-q}{pq}} = L |\Omega|^{1/q}.
$$
Hence, we have a uniform bound for the sequence $(u_p)$ in every
$W^{1,q}(\Omega)$. Taking the limit as $p\to \infty$, we conclude that there exists a
function $u_\infty$ such that, up to a subsequence, $u_p \to u_\infty$, locally uniformly in $\overline{\Omega}$ and weakly in every
$W^{1,q}(\Omega)$. Clearly, $u_\infty \geq \Psi$ pointwise. Also,
$$
\left(\int_\Omega |D u_\infty |^q\right)^{1/q}
\leq L |\Omega|^{\frac{1}{q}} \qquad \forall q > 1.
$$
We then conclude that $u_\infty$ is a Lipschitz function, with
$$
    \| D u_\infty \|_{L^\infty (\Omega)} \leq L .
$$
Since this holds being $L$ the $L^\infty$-norm of the gradient of any extension
of $F$ that is above $\Psi$, we conclude that $u_\infty$ is a
solution of the minimization problem (cf. \cite{Ju})
\begin{equation}\label{infty.minimos}
    \min_{w|_{\partial \Omega} =F ; \ w\geq \Psi \mbox{ in } \Omega } \| D w \|_{L^\infty (\Omega)}.
\end{equation}

The minimizers $u_p$ are weak, and hence viscosity, solutions (see
\cite{JLM}) of the following obstacle problem:
$$
\left \{
\begin{array}{llll}
u_p (x) &=& F (x)  & \mbox{on } \partial \Omega, \\
u_p (x) &\geq& \Psi (x)  & \mbox{in } \Omega, \\
\displaystyle - \Delta_p u_p &=& 0  & \mbox{in }
\Omega \setminus A_p := \{u_p >\Psi\}, \\
\displaystyle - \Delta_p u_p  &\geq&  0  & \mbox{in }
\Omega.
\end{array}
\right.
$$

Concerning the PDE problem satisfied by $u_\infty$, we verify that it
is a viscosity solution to the obstacle problem for the infinity Laplacian:

$$
\left \{
\begin{array}{llll}
u_\infty (x) &=& F (x)  & \mbox{on } \partial \Omega, \\
u_\infty (x) &\geq& \Psi (x)  & \mbox{in } \Omega, \\
\displaystyle - \Delta_\infty u_\infty  &=& 0  & \mbox{in }
\Omega \setminus A_\infty = \{u_\infty >\Psi\}, \\
\displaystyle - \Delta_\infty u_\infty  &\geq& 0  & \mbox{in }
\Omega.
\end{array}
\right.
$$

\noindent Indeed, fix a point $y$  in  the set $\{ u_\infty > \Psi \}$. From the uniform
convergence, $u_p >\Psi$  in a neighbourhood of $y$, provided $p
\gg 1$.  Hence, taking the limit as $p\to \infty$ in the
viscosity sense, we obtain
$$
    -\Delta_\infty u_\infty =0 \quad \mbox{ in } \{ u_\infty > \Psi\}.
$$
Moreover, a uniform limit $u_\infty$ verifies
$$
    -\Delta_\infty u_\infty \geq 0 \quad \mbox{ in } \Omega
$$
in the viscosity sense, since this holds for every $u_p$.
Let us remark that the limit obtained here does not necessarily coincide with the solution of the infinity   obstacle problem obtained by direct methods in \cite{BCF}.
\medskip

A crucial issue, with striking implications, is to characterize
the limit $u_\infty$. We give two characterizations, one involving
supersolutions of the infinity Laplacian, the other making use of
appropriately defined cones. From both we will derive important
properties of the limit.

\begin{theorem} \label{char1}
The limit $u_\infty$ is the smallest continuous infinity superharmonic
function in $\Omega$ that is above the obstacle and
equals $F$ on the boundary.
\end{theorem}

\begin{proof}

Let $\mathcal{F}$ be the set of all continuous functions $v$ that are
infinity super\-harmonic in $\Omega$ and satisfy
$v \geq \Psi$ in $\Omega$ and $v = F$ on $\partial \Omega$. This set is not
empty because $u_\infty \in \mathcal{F}$. Let
$$v_\infty := \inf_{v \in \mathcal{F}} v,$$
which is upper semicontinuous (as it is the infimum of continuous functions)
and infinity superharmonic in $\Omega$.
Since  $u_\infty \in \mathcal{F}$, it is obvious that
$$u_\infty \geq v_\infty \ \mbox{ in } \ \overline{\Omega}.
$$

Now, define the open set
$$ W= \left\{ x \in \Omega :  u_\infty (x) > v_\infty (x) \right\}. $$
On $\partial W \subset \overline{\Omega}$, we have $v_\infty = u_\infty$. Moreover,
$$u_\infty > v_\infty \geq \Psi \  \mbox{ in } \ W$$
so $W \subset \{u_\infty > \Psi\}$ and $u_\infty$ is infinity harmonic in $W$.
Thus, by the comparison principle,
$$u_\infty \leq v_\infty \ \mbox{ in }\ W, $$
a contradiction that shows that $W=\emptyset$. Consequently, $u_\infty \equiv v_\infty$.
\end{proof}

\begin{corollary}
The limit $u_\infty$ is unique.
\end{corollary}

\begin{proof}
Suppose we have two limits, say $u_{1,\infty}$ and $u_{2,\infty}$. Then
$$v = u_{1,\infty} \wedge u_{2,\infty}$$
is also an infinity superharmonic function in $\Omega$ that is above the obstacle and
equals $F$ on the boundary. By the theorem, we have
$$u_{i,\infty} \leq v , \ \ i=1,2$$
and since, trivially, $v  \leq u_{i,\infty}$,  $i=1,2$, we conclude  that
$$u_{1,\infty}=v=u_{2,\infty}.$$
\end{proof}

Let's now turn to our second characterization of the limit. For
this, consider the family of cones with vertex at a boundary point
and positive opening, which lie above both the obstacle and the
boundary data. For more on comparison with cones and the characterization of infinity harmonic functions see \cite{C}.

To be concrete, for $y \in \partial \Omega$ and $b=(b_1,b_2)$, with $b_1 \geq 0$, we consider the cones
$$ K_{y}^b (x) = b_1 |x-y| + b_2$$
such that
$$K_{y}^b (x) \geq F(x), \qquad x \in \partial \Omega$$
and
$$K_{y}^b (x) \geq \Psi (x), \qquad x \in \Omega.$$
Note that, since the vertex of the cone is at the boundary of $\Omega$, these
cones are infinity harmonic in $\Omega$, that is, $-\Delta_\infty
K_{y}^b =0$ in $\Omega$. We denote by ${\mathcal K}$ the family of all such cones.

Now, we define
$$K_\infty (x) := \inf_{{\mathcal K}} K_{y}^b (x), \quad x \in \overline{\Omega}.$$
It is obvious that
$$K_{\infty} (x) \geq F(x), \quad x \in \partial \Omega$$
and
$$K_{\infty} (x) \geq \Psi (x), \quad x \in \Omega.$$

\begin{proposition} \label{coimbra}
The function $K_\infty$ is Lipschitz continuous in $\overline{\Omega}$
and infinity superharmonic in $\Omega$. Moreover,
$$K_\infty (y) = F(y), \quad y \in \partial \Omega.$$
\end{proposition}

\begin{proof}
Since we assume that $F$ is Lipschitz, we have that for
every point $y\in \partial \Omega$, there exists a
constant $L$ such that, for every $b_1 >L$ and every $b_2>L$,
$$K_{y}^b (x) \geq F(x) \qquad \mbox{and} \qquad  K_{y}^b (x) \geq
\Psi (x).$$
Hence, when computing the infimum that defines $K_\infty (x)$, we
can restrict to cones with $b=(b_1,b_2)$ in a compact set and
since $y\in \partial
\Omega$ (which is also compact), we conclude that the infimum is
in fact a minimum. This means that, for every $x\in
\overline{\Omega}$, there exists a $y \in \partial \Omega$ and a
$b=(b_1,b_2)$, with $|b_i| \leq L$, depending on $x$, such that
$$K_\infty (x) = K_{y(x)}^{b(x)} (x).$$
From this fact, it follows that $K_\infty$ is Lipschitz continuous in
$\overline{\Omega}$. Let's show why. Take any two points $\hat{x}, \tilde{x} \in \overline{\Omega}$; we have
$$K_\infty (\hat{x}) = K_{y(\hat{x})}^{b(\hat{x})} (\hat{x}) \quad \mathrm{and}
\quad K_\infty (\tilde{x} ) = K_{y(\tilde{x} )}^{b(\tilde{x} )} (\tilde{x}).$$
From the definition, it is clear that $K_\infty (\hat{x})  \leq K_{y(\tilde{x} )}^{b(\tilde{x} )}
 (\hat{x})$ and thus
\begin{eqnarray*}
K_\infty (\hat{x})  - K_\infty (\tilde{x}) & \leq & K_{y(\tilde{x} )}^{b(\tilde{x} )}
(\hat{x}) - K_{y(\tilde{x} )}^{b(\tilde{x} )} (\tilde{x}) \\
& = & b_1 (\tilde{x}) \left( \left| \hat{x} - y(\tilde{x} ) \right| -
\left| \tilde{x} - y(\tilde{x} ) \right| \right)  \\
& \leq & L \left| \hat{x} - \tilde{x} \right|.
\end{eqnarray*}
Reversing the role of $\hat{x}$ and $\tilde{x}$ gives the desired Lipschitz regularity.

Moreover, as the infimum of infinity harmonic functions, $K_\infty$ is in\-fi\-nity superharmonic, \textit{i.e.},
\begin{equation} \label{Kvisc}
-\Delta_\infty K_{\infty} \geq 0\ \  \mbox{in} \ \Omega.
\end{equation}

Finally, by taking $b_1$ large enough and $b_2 = F(y)$, we also have, recalling  \eqref{obst_vs_bdry},
$$ F(y) \leq K_\infty (y) \leq K_{y}^b (y) = F(y)$$
and, hence, $K_\infty (y) = F(y)$, for $y \in \partial \Omega$.
\end{proof}

\begin{theorem}
The limit $u_\infty$ is such that
\begin{equation} \label{fortaleza}
u_\infty (x) \leq K_\infty (x), \quad x \in \overline{\Omega}.
\end{equation}
Equality holds if, and only if, $K_\infty (x)$ is infinity harmonic outside of
its coincidence set $\{ K_\infty = \Psi\}$.
\end{theorem}

\begin{proof}
Inequality \eqref{fortaleza} follows immediately from Proposition \ref{coimbra} and Theorem \ref{char1}.
If we have an equality it is also immediate that $K_\infty (x)$ is infinity harmonic
outside of its coincidence set $\{ K_\infty = \Psi\}$ So we are left to prove the other implication.

Arguing by contradiction, assume that
$$W= \{ x \in \Omega :  K_\infty (x) > u_\infty (x) \} \neq
\emptyset. $$
Note that $W$ is open because $u_\infty$ and $K_\infty$ are continuous
functions. Since $W \subset \{K_\infty
> \Psi\}$, we deduce that $-\Delta_\infty K_\infty =0$ in
$W$. But $-\Delta_\infty u_\infty \geq 0$ in $\Omega$ (thus in $W$) and $u_\infty
=K_\infty$ on $\partial W$ so, by the comparison principle for the
infinity Laplacian, we conclude that
$$u_\infty \geq K_\infty \ \ \mbox{in} \ W,$$
a contradiction that shows that $W= \emptyset$ and completes the
proof.
\end{proof}

\begin{remark}
{\rm The condition that $K_\infty (x)$ is infinity harmonic outside of
its coincidence set $\{ K_\infty = \Psi\}$ strongly depends on the
geometry of the problem. In the radial example explicitly computed
in the appendix, the condition holds. However, in general, this is not
the case, as the following example shows. Consider $\Omega$ to be
the union of two disjoints balls connected by a narrow tube of
width $\delta$, an obstacle placed in one of the balls and
boundary data $F=0$. It can be readily checked that, as $\delta
\to 0$, $u_\infty \to 0$ in the ball without obstacle. But
$K_\infty$ is uniformly bounded below inside this ball since the
opening of the corresponding cones is uniformly bounded below (as
these cones have to be above the obstacle).}
\end{remark}

\begin{corollary}
Assume the obstacle $\Psi$ is differentiable and equality holds in
\eqref{fortaleza}. Then $u_\infty$ is differentiable at the free
boundary and
$$Du_\infty (x_0) = D\Psi (x_0), \quad \forall x_0 \in \partial \{ u_\infty = \Psi \}.$$
\end{corollary}

\begin{proof}
Let $x_0 \in \partial\{ u_\infty = \Psi \}$. It follows from the
previous results that there exists a cone $K_{y_0}^b$ such that
\begin{equation} \label{alicante}
K_{y_0}^b (x_0) = K_\infty (x_0) = u_\infty (x_0) = \Psi (x_0)
\end{equation}
and
\begin{equation} \label{rambla}
K_{y_0}^b (x) \geq K_\infty (x) = u_\infty (x) \geq \Psi (x), \quad \forall x \in \Omega.
\end{equation}
Hence, $K_{y_0}^b (x)- \Psi (x)$ attains a minimum at $x_0$ and, since it is differentiable,
$$DK_{y_0}^b (x_0) = D \Psi (x_0).$$
From \eqref{alicante} and  \eqref{rambla}, we conclude that $u_\infty$ is also differentiable at $x_0$, with
$$Du_\infty (x_0) = D\Psi (x_0),$$
as claimed.
\end{proof}

\begin{remark}
{\rm As a consequence of this corollary, we conclude that $u_\infty$ is differentiable
everywhere in $\Omega$. In fact, in the interior of the
coincidence set, it coincides with the differentiable obstacle
and, in the interior of the non-coincidence set, it is infinity
harmonic, thus differentiable everywhere by the results of
\cite{ES}.
Also note that the radial solution constructed in the appendix is a $C^1$--solution that
can be characterized by the equality in \eqref{fortaleza}.}
\end{remark}

\medskip

We close this section with the analysis of the behaviour at infinity of the
coincidence sets for the $p$-obstacle problem and relate it with
the coincidence set of the limiting problem. We recall that
$$\limsup_{p \to \infty} A_p =  \bigcap_{p=1}^\infty  \, \bigcup_{n \geq p}
A_n \qquad \mathrm{and} \qquad \liminf_{p \to \infty} A_p =  \bigcup_{p=1}^\infty  \,
\bigcap_{n \geq p} A_n.$$

\begin{theorem} \label{teo.conver.coin.sets}
Assume $\Psi$ is smooth and verifies
$-\Delta_\infty \Psi >0$. Let $A_p=\{ u_p = \Psi\}$ be the coincidence sets of the
$p$-obstacle problems and $A_\infty=\{ u_\infty = \Psi\}$ be the
coincidence set of the limiting problem. Then
\begin{equation}
\overline{\mathrm{int} (A_\infty)} \subset \liminf_{p\to \infty} A_p
\subset \limsup_{p \to \infty} A_p \subset A_\infty.
\end{equation}
\end{theorem}

\begin{proof}
Given a neighborhood $V$ of $A_\infty$, $\Omega \setminus V$ is a
closed set contained in $\{ u_\infty > \Psi\}$. Thus, the
continuity of $u_\infty-\Psi$ gives us a $\eta>0$ such that
$u_\infty-\Psi > \eta$ in $\Omega \setminus V$. Using the uniform
convergence of $u_p$ to $u_\infty$, we conclude that, for $p$
large enough, we also have $u_p - \Psi >
\eta$ in $\Omega \setminus V$. Therefore, we conclude that $\Omega
\setminus V \subset \{ u_p > \Psi\}$ and, consequently, that
$$A_p \subset V,$$
for every large enough $p$. This shows that
$$
\limsup_{p \to \infty} A_p \subset V,
$$
for any neighborhood $V$ of $A_\infty$, and since $A_\infty$ is compact, we also obtain
$$
\limsup_{p \to \infty} A_p \subset A_\infty.
$$

Next, given $x_0 \in \mathrm{int} (A_\infty)$, if we have
$$
u_{p_j} (x_0) > \Psi (x_0),
$$
for a subsequence $p_j \to \infty$, then
$$
- \Delta_{p_j} u_{p_j} (x_0) = 0
$$
and, passing to the limit in the viscosity sense, we conclude that
$$
- \Delta_\infty \Psi (x_0) = - \Delta_\infty u_\infty (x_0) =0,
$$
a contradiction with $ -\Delta_\infty \Psi >0$.
Therefore, we conclude that for every $x_0 \in \mathrm{int} (A_\infty)$,
there exists $p_0=p_0 (x_0)$ such that
$$u_{n} (x_0) = \Psi (x_0),$$
for every $n \geq p_0$. This means that
$$x_0 \in \bigcap_{n \geq p_0} A_n$$
and consequently
$$\mathrm{int} (A_\infty) \subset \liminf_{p\to \infty} A_p.$$
Since the larger set is closed, we also obtain
$$\overline{\mathrm{int} (A_\infty)} \subset \liminf_{p\to \infty} A_p$$
and the proof is complete.
\end{proof}

\section{$\infty$-obstacle type problems and $C^{1,\frac{1}{3}}$--behavior at the free boundary} \label{batfb}

In this section we turn our analysis towards optimal regularity estimates along the free boundary. We shall consider here the zero obstacle type problem that consists in the analysis of a function satisfying:
\begin{eqnarray}
	u \ge 0 &\text{in}& B_1 \label{positive} \\
	\Delta_\infty u = f(x) &\text{in}& \{u > 0\}. \label{Z-PDE}
\end{eqnarray}
Equivalently, we can write the zero obstacle type problem as
\begin{equation}\label{eq Z-O}
	\min\{\Delta_\infty u - f(x), u\} = 0,
\end{equation}
which is understood in the viscosity sense.  We comment that the limiting variational problem studied in section \ref{S. limit p-obs} can be transformed into a zero obstacle-type problem by defining $v = u - \Psi$, under the assumption that $\Delta_\infty v$ is bounded. Thus, the results proven in this section apply to the variational setting, provided this extra assumption is verified.

\par

The ultimate goal  is to show that a solution to \eqref{eq Z-O} grows precisely as
$$
	\left[ \text{dist}(x, \partial \{u> 0\}) \right]^{4/3}
$$
away from the free boundary. Throughout this section, we work under the assumption that $f(x)$ is continuous and bounded away from zero and infinity, \textit{i.e.},
\begin{equation}
 0 < \nu \le f(x) \le M < \infty \label{cond f}
\end{equation}
will be enforced from this point on.
Such a condition is natural in the context of obstacle-type problems and allows us to prove existence and uniqueness for problem \eqref{eq Z-O} by means of a Perron's type method.
\begin{theorem}\label{ex ZO} Given a function $g \in C(\partial B_1)$, with $g>0$, and $f$ satisfying \eqref{cond f}, there exists a unique function $u \in C(\bar{B}_1)$, satisfying
\begin{equation} \label{obst_prob}
	\left \{
		\begin{array}{rll}
			\min\{\Delta_\infty u - f(x), u\} &=& 0  \text{ in } B_1\\
			u & = & g  \text{ on } \partial B_1
		\end{array}
	\right.
\end{equation}
in the viscosity sense. Assuming further that $f$ is uniformly Lipschitz continuous in $B_1$, then $u$ is locally Lipschitz continuous in $B_1$.
\end{theorem}

\begin{proof} The proof of existence goes along the lines of the proof of \cite[Theorem 1]{LW}. Define
\begin{equation}\label{set PM}
	\mathcal{A}_{f,g}^{+} := \left \{ v \in C(\overline{B_1}) \suchthat v\ge 0, \ \Delta_\infty v \le f(x) \text{ in } B_1 , \text{ and } v \ge g \text{ on } \partial B_1 \right \}.
\end{equation}
In the sequel, take
\begin{equation}\label{P sol}
	u(x) := \inf\limits_{v\in \mathcal{A}_{f,g}^{+}} v(x), \quad \text{for } x \in \overline{B_1}.
\end{equation}
Clearly $u \ge 0$ in $B_1$ and $\Delta_\infty u \le f(x)$ in the viscosity sense. It is also an upper-semicontinuous function and thus the set $\{u > 0\}$ is open. Given an open ball $B_\delta$, with $\overline{B_\delta} \subset \{u>0\}$, we can perform the same Perron's argument implemented in the proof of  \cite[Theorem 1]{LW}  to show that $\Delta_\infty u \ge f(x)$ in $B_\delta$. Hence,  the function $u$ defined in \eqref{P sol} does solve the obstacle-type PDE
$$
	\min\{\Delta_\infty u - f(x), u\} = 0  \text{ in } B_1.
$$
Continuity of $u$ up to the boundary follows precisely as in \cite[Theorem 1]{LW} and uniqueness is proven as in \cite[Theorem 3]{LW}.

Let us now turn our attention towards the local Lipschitz regularity of the solution $u$. Locally in $\{u > 0 \}$, $u$ satisfies $\Delta_\infty u \in L^\infty$ in the viscosity sense, thus $u$ is locally Lipschitz continuous in the non-coincidence set (see, for instance \cite[Corollary 2]{L}). Hence, such an estimate needs only to be proven near the free boundary. By continuity of $u$ and the fact that $g >0$ on $\partial B_1$, there exists a small number $\tau_0>0$ such that $u> 0$ in $B_1 \setminus B_{1-\tau_0}$. From our previous argument, there exists a constant $\Sigma > 0$, depending on $M$ and $\tau_0$, such that
\begin{equation}\label{Lip Est P sol}
	 |\nabla u(x) | < \Sigma, \quad \forall x \in  B_{1-\frac{\tau_0}{5}} \setminus B_{1-\frac{\tau_0}{10}}.
\end{equation}
For any vector $\nu$, with $|\nu| < \frac{\tau_0}{100}$, define $\sigma_\nu$ by
$$
	\sigma_\nu^3 :=  \inf\limits_{B_{1- \frac{\tau_0}{100}}} \frac{f(x)}{f(x+\nu)}.
$$
Since $f$ is strictly positive and Lipchitz continuous, it follows that
$$
	|1- \sigma_\nu| + |1- \sigma_\nu^3| \le K_0 |\nu|.
$$
In the sequel, let us label $r_0 := 1-\frac{3}{20}\tau_0$ and define  $u_\nu \colon B_{ r_0} \to \mathbb{R}$ by
$$
	u_\nu(x) := \sigma_\nu \cdot u(x+\nu) + \left (\Sigma + K_0 \sup\limits_{B_1} u\right ) \cdot |\nu|.
$$
We now apply the analysis from the beginning of this proof to the domain $B_{r_0}$. One simply verifies that $u_\nu$ belongs to the set
$$
	\tilde{\mathcal{A}}_{f,g}^{+} :=  \left \{ v \in C(\overline{{B}_{r_0}}) \suchthat v\ge 0, \ \Delta_\infty v \le f(x) \text{ in } B_{r_0} , \text{ and } v \ge u \text{ on } \partial B_{r_0} \right \}.
$$
By uniqueness, $u |_{B_{r_0}}$ is the infimum among all functions in $\tilde{\mathcal{A}}_{f,g}^{+}$. Thus, we can write, for any $x\in B_{r_0}$,
$$
	u_\nu(x) \ge u(x),
$$
which immediately yields
$$
	u(x+\nu) - u(x) \ge - \left (\Sigma + 2 K_0 \sup\limits_{B_1} u\right ) \cdot |\nu|
$$
and the local Lipschitz estimate for $u$ follows.
\end{proof}

We remark that, assuming only the boundedness of $f(x)$, the local Lipschitz continuity of the solution to the infinity obstacle problem is a consequence of the next lemma.

\begin{lemma} \label{bddinflap}
Let \eqref{cond f} be in force and let $u$ be the viscosity solution to the obstacle problem \eqref{obst_prob}. Then
$$
	|\Delta_\infty u| \le M.
$$
\end{lemma}
\begin{proof} The idea of the proof is to perform a singular approximation of the obstacle problem. Let $\zeta$ be a nonnegative real $C^1$ function satisfying $\supp \, \zeta = [0,1]$ and $\int \zeta(t) dt = 1$. For each $\epsilon > 0$, consider the boundary value problem
\begin{equation}\label{App OP}
	\left \{
		\begin{array}{rll}
			 \Delta_\infty u_\epsilon &=&  f(x) \cdot \displaystyle \int_0^{u_\epsilon/\epsilon} \zeta(t) \, dt  \ \text{ in } B_1\\
			 & & \\
			u_\epsilon  & = & g  \text{ on } \partial B_1.
		\end{array}
	\right.
\end{equation}
Notice that the reaction term
$$
	 f(x) \cdot \displaystyle \int_0^{u_\epsilon/\epsilon} \zeta(t) \, dt =:\beta(x,u_\epsilon),
$$
is monotone non-decreasing with respect to $u_\epsilon$. Hence, as before, by means of a Perron's type method (see \cite{BM, BM2}), the Dirichlet problem \eqref{App OP} can be uniquely solved. Clearly,
$$
	| \Delta_\infty u_\epsilon | \le M.
$$
Thus, it follows from Lipschitz estimates and uniform continuity up to the boundary (cf., for example, \cite[Corollary 2]{L}), that the family $\{u_\epsilon\}_{\epsilon > 0}$ is equicontinuous in $B_1$. Up to a subsequence, $u_\epsilon$ converges uniformly to a function $v$. The limiting function $v$ is nonnegative, agrees with $g$ on the boundary, and satisfies $|\Delta_\infty v|  \le M$, in the viscosity sense. In particular, $v$ is locally Lipschitz continuous in $B_1$. Now, given a point $z \in \{v > 0 \} \cap B_1$, by the triangular inequality, one easily checks that
$$
	B := B_{\frac{v(z)}{2L}}(z) \subset \left \{ v > \frac{v(z)}{2} > 0 \right  \},
$$
where $L$ is the Lipschitz norm of $v$ on $B_{1-|z|}$.  In particular
$$
	\Delta_\infty u_\epsilon = f(x) \text{ in } B,
$$
for all $\epsilon < \frac{v(z)}{2}$. By stability, we deduce that
$\Delta_\infty v = f(x)$  in  $B$ as well. Since $z \in \{v > 0 \}$ was taken arbitrary, it follows that $v$ satisfies $\Delta_\infty v = f(x)$ in $\{v > 0 \}$. We have verified that $v$ solves the same boundary value problem as $u$. Thus, by uniqueness, $u = v$ and the lemma is proven.
\end{proof}

As commented earlier, it remains unknown, up to now, whether a generic infinity harmonic functions is more regular than differentiable. Hence the gradient estimate given by Theorem \ref{ex ZO} is the best we can reach at this point. Surprisingly enough, at the free boundary, there is more.  We are now ready for our main result, which gives the optimal $C^{1, \frac{1}{3}}$-regularity estimate for solutions of the infinity obstacle problem along the free boundary.

\begin{theorem}[Sharp $C^{1,\frac{1}{3}}$--regularity at the free boundary]  Let $u$ be a solution to \eqref{eq Z-O} and  $x_0 \in \partial \{u  > 0 \}$ be a generic free boundary point. Then
\begin{equation}\label{OG}
	\sup\limits_{y \in B_r(x_0)} u(y)  \le C\, r^{4/3},
\end{equation}
for a constant $C$ that depends only upon the data of the problem.
\label{thm OG}
\end{theorem}

\begin{proof} For simplicity, and without loss of generality, assume $x_0 = 0$. By combining discrete iterative techniques and a continuous reasoning (see, for instance, \cite{CKS}), it is well established that proving estimate \eqref{OG} is equivalent to verifying the existence of a constant $C>0$, such that
\begin{equation} \label{OG-eq01}
	\mathfrak{s}_{j+1} \le \max \left \{C\, 2^{- {4}/{3} \cdot (j+1)}  , \ 2^{-4/3} \mathfrak{s}_j \right \}, \quad \forall \, j \in \mathbb{N},
\end{equation}
where
$$
	\mathfrak{s}_j = \sup\limits_{B_{2^{-j}}} u.
$$

Let us suppose, for the sake of contradiction, that \eqref{OG-eq01} fails to hold, \textit{i.e.}, that for each $k \in \mathbb{N}$, there exists $j_k \in \mathbb{N}$ such that
\begin{equation}
\mathfrak{s}_{j_k+1} > \max \left \{k\, 2^{-{4}/{3}\cdot (j_k+1)} , \ 2^{-4/3}  \mathfrak{s}_{j_k} \right \}.
\label{contra}
\end{equation}
Now, for each $k$, define the rescaled function $v_k \colon B_1 \to \mathbb{R}$ by
$$	
	v_k(x) := \frac{u(2^{-j_k}x)}{\mathfrak{s}_{j_k+1}}.
$$	
One easily verifies that
\begin{equation}
0 \le v_k(x) \le \sqrt[3]{16}, \quad \forall x\in B_1;
\label{OG-eq03}
\end{equation}	
\begin{equation}
v_k(0) = 0;
\label{OG-eq04}
\end{equation}	
\begin{equation}
\sup\limits_{B_{\frac{1}{2}}} v_k = 1. \label{OG-eq06}
\end{equation}	
Moreover, we formally have
\begin{eqnarray*}
\Delta_\infty v_k (x) & = & \frac{2^{-j_k}}{\mathfrak{s}_{j_k+1}} D u (2^{-j_k}x) \cdot \left( \frac{2^{-2 j_k}}{\mathfrak{s}_{j_k+1}} D^2 u (2^{-j_k}x)  \right) \cdot \frac{2^{-j_k}}{\mathfrak{s}_{j_k+1}} D u (2^{-j_k}x)\\
 & = & \frac{2^{-4j_k}}{\mathfrak{s}^3_{j_k+1}} \Delta_\infty u (2^{-j_k}x) =:f_k.
\end{eqnarray*}
It is a matter of routine to rigorously justify the above calculations using the language of viscosity solutions (see, \textit{e.g.}, \cite[section 2]{T}). We estimate
\begin{equation}
|f_k|  \le \frac{2^{-4j_k}}{2^{-4(j_k+1)} \, k^3} \, M = \frac{16M}{k^3} \leq 16M,
\label{OG-eq05}
\end{equation}	
using Lemma \ref{bddinflap} and \eqref{contra}.

Combining the uniform bounds \eqref{OG-eq03}, \eqref{OG-eq05}, and local Lipschitz regularity results for the inhomogeneous infinity Laplace equation (cf., for example, \cite[Corollary 2]{L}), we obtain both the equiboundedness and the equicontinuity of the sequence $(v_k)_k$. By Ascoli's theorem, and passing to a subsequence if need be, we conclude that $v_k$ converges locally uniformly to a infinity harmonic function $v_\infty$ in $B_1$ (observe that $f_k \rightarrow 0$) such that
$$	0\le v_\infty \le \sqrt[3]{16} \quad \text{and} \quad v_\infty(0) = 0.$$
We now use Harnack's inequality for infinity harmonic functions (see \cite[Corollary 2]{LM}) to obtain the bound
$$v_\infty (x) \leq e^{2|x|} \, v_\infty (0) = 0, \quad \forall \, x \in B_{1/2}.$$
It follows that $v_\infty \equiv 0$ in $B_{1/2}$, which contradicts \eqref{OG-eq06}. The theorem is proven.
\end{proof}

As a first consequence we improve the local Lipschitz regularity estimate provided by Theorem \ref{ex ZO}, where $f$ needs only to satisfy \eqref{cond f}. Indeed we obtain a finer gradient control near the free boundary.

\begin{corollary}\label{Lip} Let $u$ be a solution to \eqref{eq Z-O} in $B_1$. Then $u$ is locally Lipschitz continuous and for any point $z \in \{u>0\} \cap B_1$, there holds
$$
	|\nabla u(z)| \le C \dist(z, \partial \{u>0\})^{1/3}.
$$

\end{corollary}
\begin{proof}
Fix $z \in \{u> 0\} \cap B_{1/2}$ and label $d := \dist (z, \partial \{u>0\})$. Let $\zeta \in \partial \{u>0\}$ be a free boundary point satisfying
$$
	|\zeta - z| = d.
$$
From the $C^{1, \frac{1}{3}}$-smoothness of $u$ at $\zeta$, we know
\begin{equation}\label{Lip est01}
	 \sup\limits_{B_d(z)} u \le \sup\limits_{B_{2d}(\zeta)} u \le C \cdot d^{4/3}.
\end{equation}
We now define the auxiliary function $v\colon B_1 \to \mathbb{R}_{+}$, by
$$
	v(x) := \dfrac{u(z+ dx)}{d^{4/3}}.
$$
As argued before, $v$ satisfies
\begin{equation}\label{Lip est02}
	\Delta_\infty v = f(z+ dx), \quad \text{ in } B_1.
\end{equation}
From \eqref{Lip est01} we can estimate
\begin{equation}\label{Lip est03}
	\sup\limits_{B_{1}} v \le C.
\end{equation}
Finally, applying the gradient estimate for bounded solutions to \eqref{Lip est02}, we conclude
$$
	|\nabla v(0)| = d^{-1/3} |\nabla u(z)| \le C_2,
$$
and the Corollary is proven.
\end{proof}

Our next theorem establishes a $C^{1,\frac{1}{3}}$--estimate from below, which implies that $u$ leaves the zero-obstacle trapped by the graph of two functions of the order $\text{dist}^{4/3}(x, \partial \{u > 0\})$.

\begin{theorem} Let $u$ be a viscosity solution to \eqref{eq Z-O} and $y_0 \in \overline{\{u > 0\}}$ be a generic point in the closure of the non-coincidence set. Then
$$
	\sup\limits_{B_r(y_0)} u \ge c\, r^{4/3},
$$
for a constant $c>0$ that depends only upon $\nu$.
\label{thm ND}
\end{theorem}

\begin{proof} By continuity arguments, it is enough to prove the result for points in the non-coincidence set. For simplicity, and without loss of generality, take $y_0 =0$. Define the barrier
$$
	\mathcal{B}_\infty(x) := \frac{3}{4} \sqrt[3]{3 \nu}\, |x|^{4/3},
$$	
which satisfies, by direct computation,
$$
	\Delta_\infty \mathcal{B}_\infty = \nu.
$$
Hence,
$$
    \Delta_\infty  u = f(x) \ge \nu = \Delta_\infty \mathcal{B}_\infty, \quad \text{in } \{u > 0\},
$$
in the viscosity sense. On the other hand,
$$
  u \equiv 0 < \mathcal{B}_\infty \quad \text{on } \partial \{u >  0  \} \cap B_r.
$$
Therefore, for some point $ y^\star \in \partial B_r \cap \{u > 0 \}$, there must hold
\begin{equation}\label{ND eq05}
    u(y^\star) >  \mathcal{B}_\infty(y^\star);
\end{equation}
otherwise, by Jensen's comparison principle for infinity harmonic functions \cite{J}, we would have, in particular,
$$
    0< u(0) \leq \mathcal{B}_\infty(0) = 0.
$$
Estimate \eqref{ND eq05} implies the thesis of the theorem.
 \end{proof}

As usual, as soon as we establish the precise sharp asymptotic behavior for a given free boundary problem, it becomes possible to obtain certain weak geometric properties of the phases. We conclude this section by proving that the region where the membrane is above the obstacle has uniform positive density along the free boundary, which is then inhibited to develop cusps pointing inwards to the coincidence set.

\begin{corollary}\label{cor meas} Let $u$ be a solution to \eqref{eq Z-O} and $x_0 \in \partial \{u >  0 \}$ be a free boundary point. Then
$$
    \Leb\left (B_\rho(x_0) \cap \{u >  0 \} \right ) \ge \delta_\star \rho^n,
$$
for a constant $\delta_\star > 0$ that depends only upon the data of the problem.
\end{corollary}

\begin{proof}
It follows from Theorem \ref{thm ND} that there exists a point
$$
	z \in \partial B_{\rho}(x_0) \cap \{u  >  0 \}
$$
such that
$ u(z) \ge c\, \rho^{4/3}$. By $C^{1,\frac{1}{3}}$--bounds along the free boundary, Theorem \ref{thm OG},  it
follows that
$$
	B_{\lambda \rho}(z) \subset \{u > 0\},
$$
where the constant
$$
	\lambda :=  \sqrt[4]{\left(\frac{c}{2C} \right)^3}
$$
depends only on the data of the problem.
In fact, if this were not true, there would exist a free boundary point $y \in B_{\lambda \rho}(z)$. From \eqref{OG}, we would reach
$$
	c\, \rho^{4/3} \leq u(z) \leq \sup\limits_{B_{\lambda \rho}(y)} u \le C\, (\lambda \rho)^{4/3} = \frac{1}{2} c\, \rho^{4/3},
$$
which is a contradiction. Thus,
    $$
        B_\rho(x_0) \cap B_{\lambda \rho}(z)   \subset B_\rho(x_0) \cap \{u> 0 \}
    $$
and, finally,
    $$
        \Leb\left (B_\rho(x_0) \cap \{u > 0 \} \right ) \ge
        \Leb \left (  B_\rho(x_0) \cap B_{\lambda \rho}(z) \right ) \ge \delta_\star \rho^n,
    $$
and the corollary is proven.
\end{proof}

We conclude by remarking that the thesis of Corollary \ref{cor meas} implies that the free boundary $\partial \{u>0\}$ is porous, with porosity constant $\tau>0$ that depends only on the data of the problem. In particular, the Hausdorff dimension of the free boundary is strictly less than $n$ and hence it has Lebesgue measure zero. 

 \begin{center}
{\sc Appendix: A radial explicit example}
 \end{center}

\medskip

In this appendix we construct a radially symmetric explicit solution to a (variational) obstacle problem, by means of a limiting process, namely,  taking $p \to \infty$.
For that, let us consider the $p$-obstacle problem in $B_2 \subset
\mathbb{R}^d$, with zero boundary data and the spherical cap
$\psi(x) = 1- |x|^2$ as the obstacle. It is formulated as the following minimization problem:
$$
    \text{Min } \left \{ \int_{B_2} |D v(x)|^p dx \suchthat
    v\in W^{1,p}_0(B_2) \text{ and } v(x) \ge \psi(x) \right \}.
$$

As mentioned in section \ref{S. limit p-obs},  the problem admits a unique
minimizer $u_p$. By symmetry, we conclude $u_p$ is radially
symmetric, \textit{i.e.}, $u_p(x) = u_p(|x|)$. By the geometry of the
obstacle problem, as well as its regularity theory, we know that
there exists an $h = h(p,d)$, that depends on $p$ and dimension,
such that
$$
    \left \{ \begin{array}{lll}
        u_p(x) =  \psi (x) &\text{ in } & |x|\le h, \\
        \Delta_p u_p = 0  &\text{ in } & 2> |x| > h, \\
        u_p \in C^{1,\alpha_p}   &\text{ in } &  B_2, \\
        \|D u_p \|_{L^\infty(B_\rho)} \le C(\rho,d),
    \end{array} \right.
$$
for a constant $C(\rho, d)$, which is independent of $p$. In particular, as
observed in the main text, up to a subsequence, $u_p$ converges locally
uniformly to a function $u_\infty$. Furthermore, $u_\infty$ solves
$\Delta_\infty u_\infty = 0$ within $\{u_\infty > \psi \}$ in the
viscosity sense.

Our goal is to solve the $p$-obstacle problem explicitly and then
analyze the limiting function $u_\infty$.  We are initially led to search for $p$-harmonic
radially symmetric functions. If $g(x) = f(r)$, then
$$
    \Delta_p g = |f'(r)|^{p-2} \left \{(p-1) f''(r) + \dfrac{d-1}{r} f'(r) \right \}.
$$
Solving the homogeneous ODE, we obtain
$$
    f(r) = \left \{ \begin{array}{lll}
        a + b \cdot r^{\frac{1-d}{p-1} + 1} &\text{if}&  p \not = d, \\
        a + b \cdot \ln r  &\text{if}&  p = d,
    \end{array} \right.
$$
for any constants $a, b \in \mathbb{R}$.  Returning to the
obstacle problem (we will only deal with the case, $p \not = d >
1$, as we are interested in the limiting problem as $p \to
\infty$), by regularity considerations, we end up with the
following system of equations:
\begin{equation}
a + b \cdot h^{-\alpha + 1}  = 1 - h^2  \quad \text{and} \quad  b  \cdot (-\alpha+1) h^{\alpha} = - 2h, \tag{1} \label{1}
\end{equation}
where the exponent $\alpha = \alpha (p)$ is given by
\begin{equation}
    \alpha (p)= \frac{d-1}{p-1} \longrightarrow 0 \quad \textrm{as} \quad p \to \infty. \tag{2} \label{2}
\end{equation}
The first equation in \eqref{1} comes from
continuity and the second from $C^1$--estimates. By the boundary
condition, we have
$$
    a+ b \cdot 2^{-\alpha +1} =0.
$$
Subtracting the first equality from the above equation, we obtain
$$
    b \cdot (2^{-\alpha +1} - h^{-\alpha +1}) = -1 + h^2,
$$
which simplifies out to
$$
(-\alpha +1) b \cdot h^{-\alpha} = -2 h.
$$
Combining the above with the second equation in \eqref{1}, we end up with
$$
\frac{2}{1-\alpha} (2^{-\alpha +1}h^{1+\alpha } - h^2) = 1-h^2,
$$
that is,
$$
\left(\frac{2}{1-\alpha} -1\right) h^2 - 4 \left(\frac{2^{-\alpha}}{1-\alpha} \right)
h^{1+\alpha} +1 =0.
$$
Now, we observe that, from \eqref{2}, this equation
converges to $h^2 - 4  h +1 =0$, which has as solution in $(0,1)$ (the free boundary must lie
in this interval) $h_\infty = 2-\sqrt{3}$.
With this limit, we can also compute the limit of
$$
f_p (r) = a_p + b_p r^{- \frac{d-1}{p-1} +1} = a_p + b_p
r^{-\alpha (p) +1}
$$
that is given by
$$
    f_\infty (r) = a_\infty + b_\infty r,
$$
with
$a_\infty = 4 h_\infty$ and $b_\infty = - 2 h_\infty$.
Note that $f_\infty (r) $ is infinity harmonic in $B_2 \setminus
B_{h_\infty}$ and verifies
$$
    f_\infty (h_\infty) = 1- h_\infty^2 \qquad \textrm{and} \qquad f_\infty ' (h_\infty) = -2 h_\infty.
$$
It is the solution of the limit obstacle problem.

\medskip

To conclude, it might be interesting to observe that  the solution constructed here behaves linearly along the free boundary -- and not as a $C^{1, 1/3}$ graph. This fact elucidates as to why condition \eqref{cond f} ought to be enforced so that solutions do leave the obstacle precisely as $\text{dist}^{4/3}$.
\bigskip

\noindent \textbf{Acknowledgments}. {\small JDR partially supported by DGICYT grant PB94-0153 MICINN, Spain. ET partially supported by CNPq-Brazil. JMU partially supported by FCT projects PTDC/MAT/098060/2008, UTAustin/MAT/0035/2008, UTA-CMU/MAT/0007/2009 and PTDC/MAT-CAL/0749/2012, and by CMUC, funded by the European Regional Development Fund through the program COMPETE and by the Portuguese Government through FCT under the project PEst-C/MAT/UI0324/2011.}

\medskip

{\scriptsize

\noindent {\sc Julio D. Rossi\\ Department of Mathematical Analysis, University of Alicante\\ 03080 Alicante, Spain.}

\noindent \textit{E-mail address:}  {\tt julio.rossi@ua.es}

\medskip

\noindent {\sc Eduardo V. Teixeira\\ Universidade Federal do Cear{\'a}\\ Campus of Pici - Bloco 914, Fortaleza - Cear{\'a} - 60.455-760, Brazil.}

\noindent \textit{E-mail address:} {\tt teixeira@mat.ufc.br}

\medskip

\noindent {\sc Jos{\'e} Miguel Urbano\\ CMUC, Department of Mathematics, University of Coimbra\\ 3001-501 Coimbra, Portugal.}

\noindent \textit{E-mail address:}  {\tt jmurb@mat.uc.pt}

}

\end{document}